\documentclass{amsart}
\usepackage{amsmath,amsthm,amssymb}
\usepackage{lineno,hyperref}
\modulolinenumbers[5]
\usepackage[numbers,sort&compress]{natbib}

\newtheorem{theorem}{Theorem}[section]

\newtheorem{lemma}{Lemma}[section]

\allowdisplaybreaks

\allowdisplaybreaks
\nonstopmode \numberwithin{equation}{section}

\begin{document}

\title{The Modified Lommel functions: monotonic pattern and inequalities}
%\tnotetext[Modified Lommel functions]{Modified Lommel functions}

\author[Saiful R. Mondal]{Saiful R. Mondal}

\address{Department of Mathematics and Statistics, Collage of Science,
King Faisal University, Al-Hasa 31982, Saudi Arabia}
\email{smondal@kfu.edu.sa}


        % AMS-2010 subj class. The list can be found on http://www.ams.org/mathscinet/msc/msc2010.html

\begin{abstract}
This article studies the monotonicity, log-convexity of the modified Lommel functions
by using its power series  and infinite product representation. Same properties for the ratio of the modified Lommel functions with the Lommel function, $\sinh$ and $\cosh$ are also discussed. As consequence  some Tur\'an type and reverse Tur\'an type inequalities are given. A Rayleigh type function for the
Lommel functions is derived and as an application  we obtain the Redheffer-type inequality.
\end{abstract}

\maketitle

%\linenumbers

\section{Introduction}
The Lommel functions \cite{Lommel-1, Lommel-2}  are the particular solution of the inhomogeneous Bessel differential equations
\begin{align}\label{eqn;Lommel-dif}
x^2 \mathtt{f}_{\mu, \nu}''(x)+x \mathtt{f}_{\mu, \nu}'(x)-(\nu^2-x^2)\mathtt{f}_{\mu, \nu}(x)=x^{\mu+1},
\end{align}
which are usually denoted as $s_{\mu, \nu}$ and $S_{\mu, \nu}$ \cite{W, AB} and $\mathbf{S}_{\mu, \nu}$ \cite{Szym} given by
\begin{align}
\mathtt{f}_{\mu, \nu}(x):=
\left\{
\begin{array}{llll}
\mathtt{S}_{\mu, \nu}(x)&=\frac{x^{\mu+1}}{(\mu+1)^2-\nu^2} {}_1F_{2}\left(1; \tfrac{\mu-\nu+3}{2},  \tfrac{\mu+\nu+3}{2}; -\tfrac{x^2}{4}\right)\\
\mathbb{S}_{\mu, \nu}(x)&=\mathtt{S}_{\mu, \nu}(x) + \frac{(2i)^{\mu-1}}{i^{\nu}} \Gamma(\tfrac{\mu-\nu+1}{2})\Gamma(\tfrac{\mu+\nu+1}{2})J_{\nu}(x)\\
{S}_{\mu, \nu}(x)&=\mathbb{S}_{\mu, \nu}(x) + i 2^{\mu-1} \cos\left(\tfrac{\pi(\mu-\nu)}{2}\right) \Gamma(\tfrac{\mu-\nu+1}{2})\Gamma(\tfrac{\mu+\nu+1}{2})H_{\nu}^{(1)}(x),
\end{array}\right.
\end{align}
where $J_{\nu}$ and  $H^{(1)}_{\nu}$ are respectively the Bessel and the Hankel function of the first kind. The above functions satisfy the recurrence relation
\begin{align}\label{rec-L}
\mathtt{f}_{\mu+2, \nu}(x)=x^{\mu+1}- \left((\mu+1)^2-\nu^2\right) \mathtt{f}_{\mu, \nu}(x)
\end{align}

The application of the Lommel functions can be seen in various branches of mathematics and mathematical physics. The mathematical properties of the Lommel functions are available in the literature \cite{RM, Glasser, Cooke,Szym,Steinig,Pidduck,CY,KM}.
Like the modfied Bessel functions, the analogous  of the Lommel functions is the modified Lommel functions. This functions first appear in the theory of screw propeller \cite{Goldstein} and later analysed in \cite{Olver,Rollinger,Szym}.
The modified Lommel function $\mathtt{g}_{\mu,\nu}$ is a particular solution of the
differential equation
\begin{align}\label{eqn5}
 x^2 y''(x)+x y'(x)-(x^2+\nu^2)y(x)=x^{\mu+1},
 \end{align}
and satisfies the relation $\mathtt{g}_{\mu,\nu}(x)=i^{-(\mu+1)}\mathtt{f}_{\mu,\nu}(ix)$. Clearly, $\mathtt{g}_{\mu,\nu}$ satisfies the recurrence relation
\begin{align}\label{rec-mL}
\mathtt{g}_{\mu+2, \nu}(x)=\left((\mu+1)^2-\nu^2\right) \mathtt{g}_{\mu, \nu}(x)-x^{\mu+1}.
\end{align}

For $k \in \{0, 1, 2, \ldots\}$, consider the function
\begin{align}\label{eqn-1}
\varphi_k(x):= {}_1F_{2}\left(1; \tfrac{\mu-k+2}{2},  \tfrac{\mu-k+3}{2}; -\tfrac{x^2}{4}\right),
\end{align}
where $x \in \mathbb{R}$ and $\mu \in \mathbb{R}$ such that $\mu-k$ is not in $\{ 0, -1, -2, \ldots\}$. In \cite{Baricz-Kumandos}, it is shown that $\varphi_k$ is an even real entire function of order one and poses the Hadamard factorization
\begin{align}\label{eqn2}
\varphi_k(x)= \prod_{j=1}^\infty \left(1-\frac{x^2}{\eta_{\mu, k, n}^2}\right),
\end{align}
where $\pm \eta_{\mu, k, n}$ are all zeroes of  $\varphi_k$. The infinite product in \eqref{eqn2} is absolutely convergent.
%
%The Lommel function of the first kind
%\begin{align}\label{eqn3}
%\mathtt{S}_{\mu,\nu}(x)=\frac{z^{\mu+1}}{(\mu-\nu+1)(\mu+\nu+1)}{}_1F_{2}\left(1; \frac{\mu-\nu+3}{2}; \frac{\mu+\nu+3}{2}; -\frac{x^2}{4}\right),
%\end{align}  is the solution of the non-homogenous Bessel differential equation
%\[ x^2 y''(x)+z y'(x)+(x^2-\nu^2)y(x)=z^{\mu+1}.\]
%Here $\mu>-1$ and $\nu \in \mathbb{R}$ such that $\mu \pm \nu$ are not negative odd integer.
%
The function $\varphi_k$ have close association with the  Lommel function $\mathtt{S}_{\mu,\nu}$ as
\begin{align}\label{lommel-1} \mathtt{S}_{\mu-k-1/2,1/2}(x)= \frac{x^{\mu-k+1/2}}{(\mu-k)(\mu-k+1)}\varphi_k(x).\end{align}
For $\mu \in (0,1)$, it is shown in \cite{KM} that $\mathtt{S}_{\mu-1/2,1/2}$ has only one zero in each of the interval
\[ I_{2n-1}(\mu)= \left( \left(2n-1+\frac{\mu}{2}\right)\pi, \left(2n-1+{\mu}\right)\pi\right), \quad \text{and} \quad
I_{2n}(\mu)= \left( 2n\pi, \left(2n+\frac{\mu}{2}\right)\pi\right).\]

In this article we consider the function $\mathtt{L}_{\mu,\nu}$ as
\begin{align}\label{eqn4}
\mathtt{L}_{\mu,\nu}(x):=i^{-(\mu+1)}\mathtt{S}_{\mu,\nu}(ix)=\frac{x^{\mu+1}}{(\mu-\nu+1)(\mu+\nu+1)}{}_1F_{2}\left(1; \frac{\mu-\nu+3}{2}; \frac{\mu+\nu+3}{2}; \frac{x^2}{4}\right).
\end{align}
The  function $\mathtt{L}_{\mu,\nu}$ is known as the modified Lommel function. We also consider the normalized modified Lommel functions as
\begin{align}\label{modi-lommel}
\lambda_{\mu, \nu}(x)= (\mu-\nu+1)(\mu+\nu+1) x^{-\mu-1} \mathtt{L}_{\mu,\nu}(x)
=\sum_{n=0}^\infty \frac{x^{2n}}{\left(\frac{\mu-\nu+3}{2}\right)_n \left(\frac{\mu+\nu+3}{2}\right)_n 4^n}.
\end{align}
More details about the modified Lommel functions can be seen in
\cite{Rollinger,W,ZS}

The Section $\ref{sec-2}$ in this article is devoted for the investigation of the monotonicity properties of $\lambda_{\mu, \nu}$.  Log-concavity and log-convexity properties in terms of the parameters $\mu$ and variable $x$ are also investigated.
As a consequence, direct and reverse Tur\'an-type inequalities are obtained. The ratio of the derivatives of $\lambda_{\mu, \nu}$ with $\sinh$ and $\cosh$ also considered in this section.

In Section $\ref{sec-3}$, the  special case for the Lommel and the modified Lommel functions related to $\varphi_k$ are considered.  This section investigate the monotonicity and log-convexity for the product and the ratio of the Lommel and the modified Lommel  functions. At the end a Redheffer-type inequality for both the Lommel and the modified Lommel  functions is derived.

Following lemma is required in sequel.
\begin{lemma}\label{lemma:1}\cite{Biernacki-Krzy}
Suppose $f(x)=\sum_{k=0}^\infty a_k x^k$ and $g(x)=\sum_{k=0}^\infty b_k x^k$, where $a_k \in \mathbb{R}$ and $b_k > 0$ for all $k$. Further suppose that both series converge on $|x|<r$. If the sequence $\{a_k/b_k\}_{k\geq 0}$ is increasing (or decreasing), then the function $x \mapsto f(x)/g(x)$ is also increasing (or decreasing) on $(0,r)$.
\end{lemma}
Evidently, the above lemma also holds true when both $f$ and  $g$ are even, or both are odd functions.
\section{Monotonicity pattern}\label{sec-2}
\begin{theorem}\label{thm-1} Suppose that $\mu, \mu_1>-1$ and $\nu, \nu_1 \in \mathbb{R}$ such that $\mu\pm \nu$ and $\mu_1\pm \nu_1$ are not negative odd integer. Then the following assertion are true.
\begin{enumerate}
\item[(i)] Suppose that $\mu_1 \geq \mu>-1$ and
$ (\mu_1-\mu)(\mu_1+\mu+6) \geq \nu_1^2-\nu^2.$
Then, the function $x \mapsto  \lambda_{\mu, \nu}(x)/\lambda_{\mu_1, \nu_1}(x)$ is increase on $(0, \infty)$.
\item[(ii)] If $\mu\pm \nu+3 >0$, then the function  $\mu \mapsto \lambda_{\mu, \nu}(x)$ is decreasing and log-convex on $(-1, \infty)$ for each fixed $\nu \in \mathbb{R}$ and $x >0$.  %Consequently, the function $\mu \mapsto \lambda_{\mu+1, \nu}(x)/\lambda_{\mu, \nu}(x)$ is increasing.
\item[(iii)] If $\mu\pm \nu+3 >0$, then the function  $\nu \mapsto \lambda_{\mu, \nu}(x)$ is log-convex on $\mathbb{R}$ for each fixed $\mu>-1$ and $x >0$.
    %Consequently, the function $\nu \mapsto \lambda_{\mu, \nu+1}(x)/\lambda_{\mu, \nu}(x)$ is increasing.
\item[(iv)] The function $ x \mapsto \lambda^{2k}_{\mu, \nu}(x)/\cosh(x)$ is  strictly decreasing if $(\mu-\nu+3)(\mu+\nu+3)>2$.
\item[(v)] The function $ x \mapsto \lambda^{2k+1}_{\mu, \nu}(x)/\sinh(x)$ is  strictly decreasing provided  $(\mu-\nu+5)(\mu+\nu+5)>12$.
\end{enumerate}
\end{theorem}
\begin{proof}
First consider  a sequence $\{w_n\}$ defined by
\[ w_n:= \frac{(a-b)_n(a+b)_n} {(c-d)_n(c+d)_n},\]
where $a, b, c, d$ are real numbers such that $a \pm b $ and $c \pm d$ are not negative integers or zero.

Then a  calculation yield
\begin{align*}
\frac{w_{n+1}}{w_n}= \frac{(a-b+n)(a+b+n)}{(c-d+n)(c+d+n)}\geq 1,
\end{align*}
 provided $a^2-b^2+2an+n^2 \geq c^2-d^2+2cn+n^2$, which is equivalent
 to $$2(a-c)n +a^2-b^2-c^2+d^2 \geq 0.$$ The last inequality holds for all $n \geq 0$ if $a\geq c$ and $a^2-b^2-c^2+d^2 \geq 0$.

Choose $a=(\mu_1+3)/2$, $b=\nu_{1}/2$, $c=(\mu+3)/2$ and $d=\nu/2$. Then, $a\geq c$ is equivalent to $\mu_1 \geq \mu$ and $a^2-b^2-c^2+d^2 \geq 0$ reduces to
 $(\mu_1-\mu)(\mu_1+\mu+6) >\nu_1^2-\nu^2.$ This establish the fact that under the hypothesis in $(i)$ the sequence $\{w_n\}$ is increasing. Since in this case $\{w_n\}$ represent the ratio of the coefficients of  $\lambda_{\mu, \nu}(x)$ and $\lambda_{\mu_1, \nu_1}(x)$, the result in $(i)$ follows, in view of the Lemma $\ref{lemma:1}$.

 Two prove $(ii)$ and $(iii)$, consider the function
\[g_n(\mu, \nu):=\frac{\Gamma\left(\frac{\mu-\nu+3}{2}\right) \Gamma\left(\frac{\mu+\nu+3}{2}\right)}{\Gamma\left(\frac{\mu-\nu+3}{2}+n\right) \Gamma\left(\frac{\mu+\nu+3}{2}+n\right)}.\]

The first and second partial differentiation  of $\log(g_n(\mu, \nu))$ with respect to $\mu$,
\begin{align*}
\frac{\partial}{\partial \mu}\log(g_n(\mu, \nu))&=\frac{1}{2}\left(\Psi\left(\tfrac{\mu-\nu+3}{2}\right)
+\Psi\left(\tfrac{\mu+\nu+3}{2}\right)-\Psi\left(\tfrac{\mu-\nu+3}{2}+n\right)
-\Psi\left(\tfrac{\mu+\nu+3}{2}+n\right)\right),\\
\frac{\partial^2}{\partial \mu^2}\log(g_n(\mu, \nu))&=\frac{1}{4}\left(\Psi'\left(\tfrac{\mu-\nu+3}{2}\right)
+\Psi'\left(\tfrac{\mu+\nu+3}{2}\right)-\Psi'\left(\tfrac{\mu-\nu+3}{2}+n\right)
-\Psi'\left(\tfrac{\mu+\nu+3}{2}+n\right)\right)\\&=\frac{\partial^2}{\partial \nu^2}\log(g_n(\mu, \nu)).
\end{align*}
Here, $\Psi(x)=\Gamma'(x)/\Gamma(x)$ is the digamma function which is increasing and concave on $(0, \infty)$. Thus
\[\frac{\frac{\partial}{\partial \mu}g_n(\mu, \nu)}{g_n(\mu, \nu)}=\frac{\partial}{\partial \mu}\log(g_n(\mu, \nu))<0 \quad \text{and} \quad
\frac{\partial^2}{\partial \mu^2}\log(g_n(\mu, \nu))=\frac{\partial^2}{\partial \nu^2}\log(g_n(\mu, \nu))\geq 0.\]
This conclude that $\mu \mapsto \lambda_{\mu, \nu}(x)$ is decreasing and log-convex on $(-1, \infty)$. Also, $\nu \mapsto \lambda_{\mu, \nu}(x)$ is  log-convex on $\mathbb{R}$ for each fixed $\mu>-1$ and $x \in \mathbb{R}$.  This prove (ii) and (iii) in view of the fact that the sum of log-convex functions is also log-convex.

A computation yield
\begin{align}
\lambda_{\mu, \nu}^{(2k)}(x)=\sum_{n=0}^\infty \frac{(2n+2k)!}{\left(\frac{\mu-\nu+3}{2}\right)_{n+k} \left(\frac{\mu+\nu+3}{2}\right)_{n+k}
4^{n+k} (2n)!} x^{2n}.
\end{align}
It is well-known that
\[ \cosh(x)=\sum_{n=0}^\infty \frac{x^{2n}}{(2n)!}. \]
In view of Lemma \ref{lemma:1}, it is enough to know the monotonicity of  the sequence $\{\alpha_n\}_{n \geq 0}$ where
\[ \alpha_n=\frac{(2n+2k)!}{\left(\frac{\mu-\nu+3}{2}\right)_n \left(\frac{\mu+\nu+3}{2}\right)_n
4^{n+k} }. \]
Now, for all $n \geq 0$ and $k\geq 0$, the ratio
\[ \frac{\alpha_{n+1}}{\alpha_n}=\frac{(2n+2k+2)(2n+2k+1)}{4\left(\frac{\mu-\nu+3}{2}+n+k
\right) \left(\frac{\mu+\nu+3}{2}+n+k\right)}<1,\]
provided $(\mu-\nu+3)(\mu+\nu+3) > 2$.

Similarly
\begin{align}
\lambda_{\mu, \nu}^{(2k+1)}(x)=\sum_{n=0}^\infty \frac{(2n+2k+2)! x^{2n+1}}{\left(\frac{\mu-\nu+3}{2}\right)_{n+k+1} \left(\frac{\mu+\nu+3}{2}\right)_{n+k+1}
4^{n+k+1} (2n+1)!} \quad \text{and} \quad
\sin(x)=\sum_{n=0}^\infty \frac{x^{2n}}{(2n)!}, \end{align}
together with Lemma \ref{lemma:1} yields that $\lambda_{\mu, \nu}^{(2k+1)}(x)/ \sin(x)$ is decreasing if the sequence $\{\beta_n\}_{n \geq 0}$ where
\[ \beta_n=\frac{(2n+2k+2)!}{\left(\frac{\mu-\nu+3}{2}\right)_{n+k+1} \left(\frac{\mu+\nu+3}{2}\right)_{n+k+1}
4^{n+k+1} (2n+1)!}, \]
is also decreasing. Again for all $n \geq 0$ and $k\geq 0$, the ratio
\[ \frac{\beta_{n+1}}{\beta_n}=\frac{(2n+2k+4)(2n+2k+3)}{4\left(\frac{\mu-\nu+3}{2}+n+k+1
\right) \left(\frac{\mu+\nu+3}{2}+n+k+1\right)}<1,\]
provided $(\mu-\nu+5)(\mu+\nu+5) > 12$. Hence the conclusion.
\end{proof}

From Theorem \ref{thm-1}, we have few interesting consequence. For example, the log-convexity of $\mu \mapsto \lambda_{\mu, \nu}(x)$ means, for any $\alpha \in [0,1]$ and for
$\mu_1, \mu_2 >-1$,
\begin{align}
\lambda_{\alpha\mu_1+(1-\alpha)\mu_2, \nu}(x) \leq \lambda_{\mu_1, \nu}^{\alpha}(x)
\lambda_{\mu_1, \nu}^{1-\alpha}(x).
\end{align}
In particular, if  $\mu_1=\mu+a>-1$ and $\mu_2=\mu-a>-1$ for $\mu, a \in \mathbb{R}$, and $\alpha=1/2$, then the above inequality gives the reverse of the Tur\`an's type inequality for the modified Lommel functions as
 \begin{align*}
\lambda_{\mu, \nu}^2(x) \leq \lambda_{\mu+a, \nu}(x)
\lambda_{\mu-a, \nu}(x).
\end{align*}
Similarly, the log-convexity of $\nu \mapsto \lambda_{\mu, \nu}(x)$ gives
\begin{align*}
\lambda_{\mu, \nu}^2(x) \leq \lambda_{\mu, \nu+a}(x)
\lambda_{\mu, \nu-a}(x).
\end{align*}
%The next result the Tur\`an's type inequality for the modified Lommel functions is derived using the Chebyshev integral inequality \cite[p. 40]{Mi}, which states the following: suppose $f$ and $g$  are two integrable functions and monotonic in the same sense (either both decreasing or both increasing). Let $p: (a, b) \to \mathbb{R}$ be a positive integrable function. Then
%\begin{align}\label{eqn:chebyshev-1}
%\left(\int_a^b p(t) f(t) dt\right) \left(\int_a^b p(t) g(t) dt\right) \leq \left(\int_a^b p(t) dt\right) \left(\int_a^b p(t) f(t) g(t) dt\right).
%\end{align}
%The inequality in \eqref{eqn:chebyshev-1} is reversed if $f$ and $g$ are monotonic but in the opposite sense.
%
%\begin{theorem} Let $\mu >-1$ and $\nu \in \mathbb{R}$. Then the modified Lommel functions satisfies the Tur\'an type inequality
%\[\lambda_{\mu, \nu}^2(x) \geq \lambda_{\mu+1, \nu}^2(x) \lambda_{\mu-1, \nu}^2(x).\]
%
%\end{theorem}

\section{Redheffer type bound}\label{sec-3}
In this section we prove the Redheffer-type inequality for some special kind Lommel and  modified Lommel functions. From \eqref{eqn2} and \eqref{lommel-1} it follows that
\begin{align}\label{lommel-2}
\mathtt{S}_{\mu-1/2,1/2}(x)= \frac{z^{\mu+1/2}}{(\mu)(\mu+1)}\prod_{j=1}^\infty \left(1-\frac{x^2}{\eta_{\mu, 0, n}^2}\right),\end{align}
and this implies
\begin{align}\label{mod-lommel-2}
\mathtt{L}_{\mu-1/2,1/2}(x)=i^{-\mu-1/2}\mathtt{S}_{\mu-1/2,1/2}(iz)= \frac{z^{\mu+1/2}}{(\mu)(\mu+1)}\prod_{j=1}^\infty \left(1+\frac{x^2}{\eta_{\mu, 0, n}^2}\right).\end{align}

From  \eqref{modi-lommel} we have
\begin{align}\label{modi-lommel-2}
\lambda_{\mu-1/2,1/2}(x)= \mu(\mu+1) z^{-\mu-1/2} \mathtt{L}_{\mu-1/2,1/2}(x)
=\prod_{j=1}^\infty \left(1+\frac{x^2}{\eta_{\mu, 0, n}^2}\right).
\end{align}
Also consider the normalized Lommel function as
\begin{align}\label{lomel-3}
\Lambda_{\mu-1/2,1/2}(x)= \mu(\mu+1) z^{-\mu-1/2} \mathtt{S}_{\mu-1/2,1/2}(x)
=\prod_{j=1}^\infty \left(1-\frac{x^2}{\eta_{\mu, 0, n}^2}\right).
\end{align}
Applying logarithmic differentiation on \eqref{modi-lommel-2} gives
\begin{align}\label{modi-lommel-4}
\frac{ \lambda'_{\mu-1/2,1/2}(x)}{z \lambda_{\mu-1/2,1/2}(x)}
= \sum_{n=1}^\infty \frac{2}{\eta_{\mu, 0, n}^2+x^2}.
\end{align}
A calculation gives
\begin{align}\label{modi-lommel-3}
\lim_{z \to 0}\frac{\lambda_{\mu, \nu}'(x)}{z\lambda_{\mu, \nu}(x)}=\lim_{z \to 0} \frac{\sum_{n=1}^\infty \frac{2n z^{2n-2}}{\left(\frac{\mu-\nu+3}{2}\right)_n \left(\frac{\mu+\nu+3}{2}\right)_n 4^n}}{\sum_{n=0}^\infty \frac{ z^{2n}}{\left(\frac{\mu-\nu+3}{2}\right)_n \left(\frac{\mu+\nu+3}{2}\right)_n 4^n}}
=\frac{2}{(\mu+3)^2-\nu^2}.
\end{align}
Now \eqref{modi-lommel-4} and \eqref{modi-lommel-3} together give the useful identity
\begin{align}\label{identity-1}
\sum_{n=1}^\infty \frac{1}{\eta_{\mu, 0, n}^2} =\frac{4}{(2\mu+5)^2-1}=\frac{1}{(\mu+2)(\mu+3)}.
\end{align}
For simplicity in sequel,  we will use the notation $\eta_{\mu, n}$ for $\eta_{\mu, 0, n}$, the $n^{th}$ positive zero of $\varphi_0(x)$.

Next we state and proof some results involving the function $\lambda_{\mu-1/2,1/2}$ and $\Lambda_{\mu-1/2,1/2}$. The monotonic properties of l'Hospital' rule as state in the following result are useful in sequel.

 \begin{lemma}\label{lem-L'hopital}
 \cite[Lemma 2.2]{A-V-V} Suppose that $-\infty<a<b<\infty$ and $p, q: [a, b)\mapsto \infty$ are differentiable functions such that $q'(x) \neq 0$ for $x \in (a, b)$. If $p'/q'$ is increasing(decreasing) on $(a, b)$,  then so is
$(p(x)-p(a))/(q(x)-q(a))$\end{lemma}
Next we will state and proof our main result in this section.
\begin{theorem} Suppose that $\mu >-1$ and $I_{\mu}:=(-\eta_{\mu, 1}, \eta_{\mu, 1})$.
\begin{enumerate}
\item The function $x \mapsto \lambda_{\mu-1/2,1/2}(x)$ is increasing on $(0, \infty)$.
\item The function $x \mapsto \lambda_{\mu-1/2,1/2}(x)$ is strictly log-convex on $I_{\mu}$ and strictly geometrically convex on $(0, \infty)$.
\item  The modified Lommel functions  $\lambda_{\mu-1/2,1/2}(x)$ satisfies the sharp exponential Redheffer-type inequality
\begin{align}\label{redheff-ineq}
\left(\frac{\eta_{\mu,1}^2+x^2}{\eta_{\mu,1}^2-x^2} \right)^{a_{\mu}}
\leq \lambda_{\mu-1/2,1/2}(x)\leq \left(\frac{\eta_{\mu,1}^2+x^2}{\eta_{\mu,1}^2-x^2} \right)^{b_{\mu}},
\end{align}
on $I_{\mu}$. Here, $a_{\mu}=0$ and $b_{\mu}=\frac{2 \eta_{\mu, 1}^2}{(\mu+2) (\mu+3)}$ are the best possible constant.

\item The function $x \mapsto \lambda_{\mu-1/2,1/2}(x) \Lambda_{\mu-1/2,1/2}(x)$ is increasing on $(-\eta_{\mu,1}, 0]$ and decreasing on $[0, \eta_{\mu,1})$

\item The function
\[ x \mapsto \frac{\lambda_{\mu-1/2,1/2}(x)} {\Lambda_{\mu-1/2,1/2}(x)}=\frac{L_{\mu-1/2,1/2}(x)}{S_{\mu-1/2,1/2}(x)},\]
is strictly log-convex on $I_{\mu}$.

\item  The  Lommel functions  $\Lambda_{\mu-1/2,1/2}(x)$ satisfies the sharp exponential Redheffer-type inequality
\begin{align}\label{redheff-ineq}
\left(\frac{\eta_{\mu,1}^2-x^2}{\eta_{\mu,1}^2} \right)^{a_{\mu}}
\leq \Lambda_{\mu-1/2,1/2}(x)\leq \left(\frac{\eta_{\mu,1}^2-x^2}{\eta_{\mu,1}^2}\right)^{b_{\mu}},
\end{align}
on $I_{\mu}$. Here, $a_{\mu}=0$ and $b_{\mu}=\frac{2 \eta_{\mu, 1}^2}{(\mu+2) (\mu+3)}$ are the best possible constant.

\end{enumerate}
\end{theorem}
\begin{proof} Consider $\mu>-1$ and $ x \in (-\eta_{\mu, 1}, \eta_{\mu, 1})$.
\begin{enumerate}
\item From \eqref{modi-lommel-4} it is evident that
\[ (\log(\lambda_{\mu-1/2,1/2}(x)))'=\frac{ \lambda'_{\mu-1/2,1/2}(x)}{\lambda_{\mu-1/2,1/2}(x)}=
\sum_{n=1}^\infty \frac{2x}{\eta_{\mu,  n}^2+x^2}>0\]
on $(0, \infty)$. Thus, for $\mu>-1$, the function $x \mapsto \log(\lambda_{\mu-1/2,1/2}(x))$ is strictly increasing on $(0, \infty)$ and consequently $x \mapsto
\lambda_{\mu-1/2,1/2}(x)$ is also increasing on $(0, \infty)$.
\item  Again from \eqref{modi-lommel-4} it follows that
\begin{align*}
\left(\frac{ \lambda'_{\mu-1/2,1/2}(x)}{\lambda_{\mu-1/2,1/2}(x)}\right)'
= \sum_{n=1}^\infty \frac{2(\eta_{\mu,  n}^2-x^2)}{(\eta_{\mu,  n}^2+x^2)^2}.
\end{align*}
Clearly, the function $x \mapsto { \lambda'_{\mu-1/2,1/2}(x)}/{\lambda_{\mu-1/2,1/2}(x)}$ is increasing for $ x \in
(-\eta_{\mu, 1}, \eta_{\mu, 1})$. This is equivalent to say that the function
$x \mapsto \lambda_{\mu-1/2,1/2}(x)$ is log-convex on $(-\eta_{\mu, 1}, \eta_{\mu, 1})$.

Another calculation from \eqref{modi-lommel-4} yields
\begin{align*}
\left(\frac{ z\lambda'_{\mu-1/2,1/2}(x)}{\lambda_{\mu-1/2,1/2}(x)}\right)'
= \sum_{n=1}^\infty \frac{4z\eta_{\mu,  n}^2}{(\eta_{\mu,  n}^2+x^2)^2}.
\end{align*}
This implies that the function $x \mapsto { x \lambda'_{\mu-1/2,1/2}(x)}/{\lambda_{\mu-1/2,1/2}(x)}$ is strictly increasing for $ x \in
(0, \infty)$ and hence
$x \mapsto \lambda_{\mu-1/2,1/2}(x)$ is geometrically convex on $(0, \infty)$.

\item  Consider the function
\[ g_{\mu}(x):= \frac{\log(\lambda_{\mu-1/2,1/2}(x))}{\log(\eta_{\mu,1}^2+x^2)-
\log(\eta_{\mu,1}^2-x^2)}.\]

Denote $p(x):=\log(\lambda_{\mu-1/2,1/2}(x))$ and $q(x):=\log(\eta_{\mu,1}^2+x^2)-
\log(\eta_{\mu,1}^2-x^2)$ on $x \in [0, \infty)$. In view of \eqref{modi-lommel-4}, it follows that
 \begin{align*}
  \frac{p'(x)}{q'(x)}=\frac{\eta_{\mu,n}^4-x^4}{4x \eta_{\mu,1}^2}\frac{\lambda_{\mu-1/2,1/2}'(x)}{\lambda_{\mu-1/2,1/2}(x)}
  =\frac{1}{2 \eta_{\mu,1}^2}\sum_{n=1}^\infty \frac{\eta_{\mu,1}^4-x^4}{\eta_{\mu,n}^2+x^2}.
 \end{align*}
and then
\begin{align*}
 \frac{d}{dx} \left(\frac{p'(x)}{q'(x)}\right)
  =-\frac{x}{ \eta_{\mu,1}^2}\sum_{n=1}^\infty \frac{z^4+2x^2\eta_{\mu,1}^2 +\eta_{\mu,n}^4}{(\eta_{\mu,n}^2+x^2)^2} \leq 0.
 \end{align*}
Thus, ${p'(x)}/{q'(x)}$  is decreasing.

Therefore,
\[ g_{\mu}(x) = \frac{p(x)-p(0)}{q(x)-q(0)}=\frac{p(x)}{q(x)}\]
is  decreasing too on $[0, \eta_{\mu,1})$ and hence
\[ a_{\mu} =\lim_{x \to \eta_{\mu,1}} g_{\mu}(x)< g_{\mu}(x)<\lim_{x \to 0} g_{\mu}(x)=b_{\mu}.\]
Finally,
\[ \lim_{x \to \eta_{\mu,1}} \frac{p'(x)}{q'(x)} =0 \quad \text{and} \quad
\lim_{x \to 0} \frac{p'(x)}{q'(x)} =\frac{\eta_{\mu,1}^2}{2 }\sum_{n=1}^\infty \frac{1}{\eta_{\mu,n}^2}=\frac{\eta_{\mu,1}^2}{2(\mu+2)(\mu+3) },\]
implies
$a_{\mu}=0$  \text{and}  $b_{\mu}={\eta_{\mu,1}^2}/(2(\mu+2)(\mu+3)).$\\
\item From \eqref{modi-lommel-3} and \eqref{lomel-3}, it is evident that
\[\lambda_{\mu-1/2,1/2}(x)  \Lambda_{\mu-1/2,1/2}(x)=\prod_{n=1}^\infty \left(1-\frac{x^4}{\eta_{\mu,n}^4}\right)\]
Thus,  by the logarithmic differentiation it follows that
\[\frac{\bigg(\lambda_{\mu-1/2,1/2}(x)  \Lambda_{\mu-1/2,1/2}(x))\bigg)'}{\lambda_{\mu-1/2,1/2}(x)  \Lambda_{\mu-1/2,1/2}(x)} = -\sum_{n=1}^\infty\frac{4 x^3}{\eta_{\mu,n}^4-x^4}. \]
Since $x \in I_\mu$, the conclusion follows.

\item From \eqref{lomel-3}, we have the logarithmic differentiation of $(\Lambda_{\mu-1/2, 1/2}(x))^{-1}$ as
    \[ \left(\log\left((\Lambda_{\mu-1/2, 1/2}(x))^{-1}\right)\right)'=\sum_{n=1}^\infty\frac{2 x}{\eta_{\mu,n}^2-x^2}\]
    \quad \text{and} \quad  \[\left(\log\left((\Lambda_{\mu-1/2, 1/2}(x))^{-1}\right)\right)''=2\sum_{n=1}^\infty\frac{ \eta_{\mu,n}^2+x^2}{(\eta_{\mu,n}^2-x^2)^2} >0. \]
    This conclude that the function $x \mapsto (\Lambda_{\mu-1/2, 1/2}(x))^{-1}$ is strictly log-convex on $I_{\mu}$.
    Finally, being the product of two strictly log-convex functions, the function
     \[ x \mapsto \frac{\lambda_{\mu-1/2,1/2}(x)} {\Lambda_{\mu-1/2,1/2}(x)}=\frac{L_{\mu-1/2,1/2}(x)}{S_{\mu-1/2,1/2}(x)},\]
    is also strictly log-convex. Note that the log-convexity of  $x \mapsto \lambda_{\mu-1/2,1/2}(x)$ follows from
     part $(2)$ of this theorem.
    \item To prove this result first we need to set up a Rayleigh type functions for the Lommel function. Define the function
        \begin{align}
        \mathbb{\alpha}^{(2m)}_{n, \mu}:=\sum_{n=1}^\infty \eta_{\mu, n}^{-2m}, \quad m=1,2, \ldots.
        \end{align}
       Logarithmic differentiation of \eqref{lomel-3} yield
       \begin{align*}
       \frac{x \Lambda'_{\mu-1/2,1/2}(x)}{\Lambda_{\mu-1/2,1/2}(x)}
       =-2\sum_{n=1}^\infty \frac{x^2}{\eta_{\mu, n}^2-x^2}
       =\sum_{n=1}^\infty \frac{x^2}{\eta_{\mu, n}^2}\left(1-\frac{x^2}{\eta_{\mu, n}^2}\right)^{-1}
       =\sum_{n=1}^\infty \frac{x^2}{\eta_{\mu, n}^2}\sum_{m=0}^\infty \frac{x^{2m}}{\eta_{\mu, n}^{2m}}.
       \end{align*}
       Interchanging the order of the summation it follows that
       \begin{align}\label{eqn-23}
       \frac{x \Lambda'_{\mu-1/2,1/2}(x)}{\Lambda_{\mu-1/2,1/2}(x)}
              =-2\sum_{m=0}^\infty \sum_{n=1}^\infty  \frac{x^{2m+2}}{\eta_{\mu, n}^{2m+2}}
              =-2\sum_{m=1}^\infty \alpha_{n,\mu}^{(2m)}  x^{2m}.
       \end{align}
Consider the function
\begin{align}\label{120}
\varphi_{\mu}(x):=\frac{\log(\Lambda_{\mu-1/2,1/2}(x))}{\log\left(1-\frac{x^{2}}
{\eta_{\mu, 1}^{2}}\right)}=\frac{\mathtt{p}_\mu(x)}{\mathtt{q}_{\mu}(x)}.
\end{align}
The binomial series together with  \eqref{eqn-23} gives  the ratio of $\mathtt{p}_\mu'$ and $\mathtt{q}_\mu'$  as
\begin{align}\label{eqn-121}
\frac{\mathtt{p}_\mu'(x)}{\mathtt{q}_\mu'(x)}=\dfrac{\frac{x \Lambda'_{\mu-1/2,1/2}(x)}{\Lambda_{\mu-1/2,1/2}(x)}}{\frac{-2 x^{2}}
{\eta_{\mu, 1}^{2}}\left(1-\frac{x^{2}}
{\eta_{\mu, 1}^{2}}\right)^{-1}}=\displaystyle\frac{\sum_{m=1}^\infty \alpha_{n,\mu}^{(2m)}  x^{2m}}{\sum_{m=1}^\infty \eta_{\mu, 1}^{-2m}  x^{2m}}.
\end{align}
Denote $d_m=\eta_{\mu, 1}^{2m}\alpha_{n,\mu}^{(2m)}$. Then
\begin{align*} {d_{m+1}}-{d_m}&= {\eta_{\mu, 1}^{2m+2}\alpha_{n,\mu}^{(2m+2)}}-{\eta_{\mu, 1}^{2m}\alpha_{n,\mu}^{(2m)}}
= \sum_{n=1}^\infty\frac{\eta_{\mu, 1}^{2m}}{\eta_{\mu, n}^{2m}}\left( \frac{\eta_{\mu, 1}^{2}}{\eta_{\mu, n }^{2}}-1\right)<0.
\end{align*}
This is equivalent to say that the sequence $\{d_m\}$ is decreasing, and hence by Lemma \ref{lemma:1} it follows that the ratio $p_{\mu}'/q_{\mu}'$ is decreasing. In view of Lemma \ref{lem-L'hopital}, we have $\varphi_\mu=p_{\mu}/q_{\mu}$ is decreasing.

From \eqref{120} and \eqref{eqn-121}, it can be shown that
\begin{align*}
\lim_{x \to 0} \varphi_{\mu}(x)&=\lim_{x \to 0} \frac{p_{\mu}'(x)}{q_\mu'(x)}
=\lim_{x \to 0} \frac{p_{\mu}''(x)}{q_\mu''(x)}=\lim_{x \to 0} \frac{p_{\mu}''(x)}{q_\mu''(x)}=\eta_{\mu,1}^2 \alpha_{\mu,n}^{(2)},\\ \notag
\quad{and}\hspace{1in}&\\
\lim_{x \to \eta_{\mu,1}} \varphi_{\mu}(x)&=\lim_{x \to \eta_{\mu,1}} \frac{p_{\mu}'(x)}{q_\mu'(x)}
=\lim_{x \to \eta_{\mu,1}}\sum_{n=1}^\infty \frac{\eta_{\mu,1}^2-x^2}{\eta_{\mu,n}^2-x^2}=1.
\end{align*}
It is easy to seen that $\eta_{\mu,1}^2 \alpha_{\mu,n}^{(2)}=b_{\mu}$. \qedhere
\end{enumerate}
\end{proof}

\end{document}